\documentclass{cmslatex}
\usepackage[paperwidth=7in, paperheight=10in, margin=.875in]{geometry}

\usepackage{amsfonts,amssymb}
\usepackage{amsmath}
\usepackage{graphicx}
\usepackage{cite}
\usepackage{enumerate}
\usepackage[backref,colorlinks,linkcolor=red,anchorcolor=green,citecolor=blue]{hyperref}

\newcommand{\E}[1]{\mathbb{E}\left[#1\right]}

\newcommand{\T}{\mathbb{T}}
\newcommand{\bbP}[1]{\mathbb{P}\left(#1\right)}
\newcommand{\norm}[1]{\left\lVert #1\right\rVert}

\renewcommand{\theequation}{\arabic{section}.\arabic{equation}}

   \allowdisplaybreaks
\begin{document}
 \title{Convergence of a Randomized Newton Method in Nonconvex Optimization \thanks{Received date, and accepted date (The correct dates will be entered by the editor).}}

          %For each author, make a block with the following macros:

          \author{Edward Huynh\thanks{edhuynh@utexas.edu}
          \and Bjorn Engquist\thanks{engquist@oden.utexas.edu}}

         \pagestyle{myheadings} \markboth{A RANDOMIZED NEWTON METHOD IN NONCONVEX OPTIMIZATION}{HUYNH, ENGQUIST} \maketitle

          \begin{abstract}
            We analyze a stochastic Newton optimization scheme for locating the unique global minimizer of a general nonconvex objective function. The method couples a Newton algorithm to additive Gaussian noise with state-dependent variance. In the bounded domain setting, we prove global almost sure convergence. The proof is based on two features of the algorithm: a nondegenerate exploratory property that ensures entrance into a neighborhood of the minimizer after a finite number of steps, and a decaying-noise property that yields contraction with high probability and prevents infinitely many exits from the neighborhood of the minimum. 
          \end{abstract}
          


\section{Introduction}\label{Sec1}
Newton’s method is one of the most common optimization algorithms and is well-known for its fast local quadratic convergence. Global optimization of non-convex objective functions is of increasing importance throughout the fields of science and engineering and is a major component in machine learning. This generates interest in randomized methods where the noise helps to avoid trapping in local minima and saddle points. The randomness typically comes from added noise or from randomized sub-sampling of the objective or loss function. 
 
Properties of stochastic Newton with variants have been studied previously, for instance, in \cite{Boyer2023, Wang2017, Roosta2019, Martin2012}. The proofs of convergence for many of these ideas rely on classical ideas from Robbins-Monro \cite{Robbins1951} or on applying the Robbins-Siegmund theorem for “almost” supermartingale type processes. In \cite{Li2022}, a general framework for a wide class of stochastic optimization methods is presented. Many of these results rely on the underlying cost function being convex or they prove almost sure convergence to a stationary point rather than to the global minimizer.
 
Convergence naturally requires the random term to vanish in the limit and for a global minimization the term cannot vanish too fast to avoid trapping in local minima. The classical papers \cite{Geman1986, Chiang1987} prove slow convergence rate with the random term decaying logarithmically with time or number of iterations for stochastic gradient descent in the general nonconvex case. The paper \cite{Chiang1987} also argues that almost sure convergence is not possible for such decay mechanism. 
 
To achieve almost sure convergence, we assume the unique optimal value of the objective function to be known so the random term can be scaled by the value of the objective function rather than by time. Knowing this value is quite common, for instance, in the classic problem of fitting a high-degree polynomial against a low number of datapoints. Similarly, the assumption holds true for optimization problems in machine learning involving overparameterized neural networks. The objective function or loss minimum is 0, although the minimum may be obtained at various minimizers. In \cite{Wojtowytsch2023, Engquist2024} convergence to the global minimizer in a nonconvex setting was shown for stochastic gradient descent based on a known or estimated minimal value and state dependent variance of the noise.
 
The rest of the paper contains the precise formulation of the convergence theorem and its proof. The proof is divided into two parts. In the exploration part we show that with any starting value the iterates will enter a small neighborhood of the global minimum after a finite number of steps with probability 1. In the so-called exploitation phase we show that starting in the neighborhood of the global minimum the iterates converge quadratically to the minimum with positive probability. Combining these properties finalizes the proof of almost sure quadratic convergence.

\section{Statement of Results and Proofs}
We consider the periodic setting $\Omega=\T^d$. The goal is to prove global almost sure convergence to the unique minimizer. The proof below is organized around two separate effects. Away from the minimizer, the noise is nondegenerate and therefore gives a uniform probability of entering any prescribed neighborhood of $x^*$. Near the minimizer, the deterministic part is the Newton map, and the state-dependent noise is small enough that there is a positive probability of remaining in a shrinking sequence of balls. Repeated entrances into the local basin then force one successful entrance almost surely.

\subsection*{Assumptions}
We assume:
\begin{itemize}
    \item[\textbf{A1}] $f\in C^2(\T^d)$ and $f$ has a unique global minimizer $x^*\in\T^d$. We normalize so that $f_{min} = f(x^*)=0$.

    \item[\textbf{A2}] There exist constants $r_0,m,L>0$ such that, on $B_{r_0}(x^*)$, one has
    \begin{gather*}
        mI_d\preceq H(x),
        \qquad
        \norm{H(x)-H(y)}\leq L\norm{x-y},
        \qquad x,y\in B_{r_0}(x^*).
    \end{gather*}
\end{itemize}
Assumption \textbf{A1} is the global nonconvex assumption and implies that for any $x\neq x^*$ such that $\nabla f(x) = 0$, we have $f(x) - f(x^*) > 0$. Assumption \textbf{A2} says that there is a sufficiently nice neighborhood of the minimizer. Throughout this section, we choose $r_0$ smaller if necessary so that the ball $B_{r_0}(x^*)$ may be identified with an ordinary Euclidean ball in a single coordinate chart of the torus. All additions below are understood modulo the periodic identification of $\T^d$.
\\
\begin{remark}\label{remark:infinitedomain}
    The assumptions of bounded domain and periodic boundary conditions are used explicitly in the exploration phase of the proofs. The boundary conditions could potentially be replaced by reflection. This will only affect the proof slightly, but the algorithm itself will be much more complex to accommodate multiple reflections. It would also be possible to handle an infinite domain but then the objective function and noise bound should be of a specific form in the far field such that the iterates will enter a fixed finite domain after a finite number of steps.
\end{remark}
\\
\begin{proposition}\label{prop:torus-quadratic-bounds}
Under \textbf{A1}--\textbf{A2}, if we define
\begin{gather*}
    M_0:=\sup_{y\in B_{r_0}(x^*)}\norm{H(y)},
\end{gather*}
then $M_0<\infty$ and
\begin{gather}\label{eq:torus-quadratic-bounds}
    \frac{m}{2}\norm{x-x^*}^2\leq f(x)\leq \frac{M_0}{2}\norm{x-x^*}^2,
    \qquad \forall x\in B_{r_0}(x^*).
\end{gather}
\end{proposition}

\begin{proof}
Since $H$ is continuous on $\overline{B_{r_0}(x^*)}$, the quantity $M_0$ is finite. Also, $\nabla f(x^*)=0$ because $x^*$ is a minimizer. Taylor's theorem with integral remainder gives
\begin{gather*}
    f(x)=\int_0^1(1-t)(x-x^*)^T H(x^*+t(x-x^*))(x-x^*)\,dt,
    \qquad x\in B_{r_0}(x^*).
\end{gather*}
The segment $x^*+t(x-x^*)$ remains inside $B_{r_0}(x^*)$. Using $mI_d\preceq H$ and $\norm{H}\leq M_0$ on this ball, we obtain
\begin{gather*}
    m\norm{x-x^*}^2
    \leq
    (x-x^*)^T H(x^*+t(x-x^*))(x-x^*)
    \leq
    M_0\norm{x-x^*}^2.
\end{gather*}
Integrating in $t$ and using $\int_0^1(1-t)\,dt=1/2$ gives the claim.
\end{proof}

We consider the stochastic iterative scheme
\begin{gather}\label{eq:torus-scheme}
    X_{n+1}= \begin{cases}
        X_n- H(X_n)^{-1}\nabla f(X_n)+ f(X_n)^\alpha Z_n, &H(X_n)\ \text{invertible},\\
        X_n + f(X_n)^\alpha Z_n, &\text{otherwise},
    \end{cases}
\end{gather}
where $\alpha>\frac12$ and $Z_n\sim N(0,I_d)$ are independent. Here the definition of the second branch of the scheme is necessary as the function $f$ may have a degenerate Hessian at (possibly) several points in its domain. However, it will turn out that this definition will only be necessary for proving entrance into a ``nice enough" neighborhood on which the Hessian will be non-singular. Note that on the torus, after the Euclidean update we take the periodic representative of the resulting point.

As an aside, the scheme \eqref{eq:torus-scheme} is reminiscent of stochastic gradient descent schemes of the form \cite{Robbins1951}
\begin{gather}\label{eq:SGDalg}
    X_{n+1} = X_n - \eta(t)G(X_{n-1}; \xi_{n})
\end{gather}
where $\{\xi_i\}$ are independent random variables, $\eta(t)$ denotes a step-size, and $G$ is some unbiased approximation of the gradient $\nabla f$. These schemes generally converge under decaying step sizes, e.g. $\eta = \eta_0/n$ \cite{Wang2023} with convergence rate $O(1/n)$ in expectation when $f$ is strongly convex. Unlike this result, we will prove Theorem \ref{thm:main-periodic} which gives a path-wise type quadratic convergence rate. Moreover, our scheme \eqref{eq:torus-scheme} will converge to the global minimizer of the (possibly) nonconvex function $f$, unlike \eqref{eq:SGDalg} which generally converges to stationary points $\tilde{x}$ (i.e. $\nabla f(\tilde{x}) = 0$).

Before proving Theorem \ref{thm:main-periodic}, we will need to prove a sequence of lemmas. We first prove an estimate on the ``deterministic" step.
\begin{lemma}[Local Newton estimate]\label{lem:torus-newton-estimate}
Define the deterministic Newton map
\begin{gather*}
    T(x):=x-H(x)^{-1}\nabla f(x)
\end{gather*}
on $B_{r_0}(x^*)$. Then for every $x\in B_{r_0}(x^*)$,
\begin{gather}\label{eq:quadratic-newton-estimate}
    \norm{T(x)-x^*}\leq \frac{L}{2m}\norm{x-x^*}^2.
\end{gather}
Consequently, for every $q\in(0,1)$ there exists $r\in(0,r_0)$ such that
\begin{gather}\label{eq:linear-newton-estimate}
    \norm{T(x)-x^*}\leq q\norm{x-x^*},
    \qquad \forall x\in B_r(x^*),
\end{gather}
and hence $T(B_r(x^*))\subset B_r(x^*)$.
\end{lemma}
\begin{proof}
The proof of \eqref{eq:quadratic-newton-estimate} follows from classical proofs in optimization texts, for instance \cite{Nocedal2006}. If $r>0$ is chosen so that $(L/(2m))r\leq q$, then \eqref{eq:linear-newton-estimate} follows immediately. The inclusion $T(B_r(x^*))\subset B_r(x^*)$ follows because $q<1$.
\end{proof}

We now choose, once and for all, a radius $r\in(0,r_0)$ small enough so that
\begin{gather}\label{eq:choice-local-radius}
    \frac{L}{2m}r\leq \frac14,
    \qquad
    f(x)\leq \frac12,\ \qquad \forall x\in B_r(x^*).
\end{gather}
The first condition gives a strict Newton contraction on $B_r(x^*)$, while the second condition is only a convenient normalization of the size of the objective in the local basin.

\begin{lemma}[Entrance into the local basin]\label{lem:torus-entrance}
Let
\begin{gather*}
    \tau_r:=\inf\{n\geq 0:X_n\in B_r(x^*)\}.
\end{gather*}
Then
\begin{gather*}
    \bbP{\tau_r<\infty}=1.
\end{gather*}
In particular, there exists $p_*>0$ such that
\begin{gather}\label{eq:entrance-geometric-tail}
    \bbP{\tau_r>n}\leq (1-p_*)^n,
    \qquad n\geq 0.
\end{gather}
\end{lemma}
\begin{proof}
By continuity of $f$ and compactness of $\T^d$, for every $0<r<r_0$ we have
\begin{gather}\label{eq:positive-gap-outside-ball}
    b_r:=\min\{f(x):x\in \T^d\setminus B_r(x^*)\}>0.
\end{gather}
The strict positivity follows from the uniqueness of the global minimizer and the normalization $f(x^*)=0$. On $\T^d\setminus B_r(x^*)$, the objective is bounded below by $b_r>0$ from \eqref{eq:positive-gap-outside-ball}. Hence the noise amplitude satisfies
\begin{gather*}
    f(y)^\alpha\geq b_r^\alpha,
    \qquad y\in \T^d\setminus B_r(x^*).
\end{gather*}
For each fixed $y\notin B_r(x^*)$, the one-step transition law has a periodized Gaussian density on $\T^d$. Moreover, since $f$ is bounded above on $\T^d$ and bounded below by $b_r$ on $\T^d\setminus B_r(x^*)$, the variance of the Gaussian perturbation is bounded above and below whenever the current state lies outside $B_r(x^*)$. Hence the corresponding periodized Gaussian density admits a strictly positive lower bound on $\T^d$, uniformly over all such current states. Therefore there exists $p_*>0$ such that
\begin{gather*}
    \inf_{y\in \T^d\setminus B_r(x^*)}
    \bbP{X_{n+1}\in B_r(x^*)\mid X_n=y}
    \geq p_*.
\end{gather*}
Now
\begin{gather*}
    \{\tau_r>n+1\}=\{\tau_r>n\}\cap \{X_{n+1}\notin B_r(x^*)\}.
\end{gather*}
Using the tower property and the Markov property,
\begin{align*}
    \bbP{\tau_r>n+1}
    &=\E{\mathbf 1_{\{\tau_r>n\}}\bbP{X_{n+1}\notin B_r(x^*)\mid \mathcal F_n}}\\
    &=\E{\mathbf 1_{\{\tau_r>n\}}\bbP{X_{n+1}\notin B_r(x^*)\mid X_n}}\\
    &\leq (1-p_*)\bbP{\tau_r>n}.
\end{align*}
Iteration gives \eqref{eq:entrance-geometric-tail}, and therefore $\bbP{\tau_r<\infty}=1$.
\end{proof}

The next lemma is the precise shrinking-ball mechanism. It says that once the process enters the local Newton basin, there is a uniformly positive probability that it remains in a deterministic sequence of smaller and smaller balls.

Before we do this, we first prove a proposition that will be needed in the proof.
\begin{proposition}[Tail bound on Gaussian Steps]\label{prop:tailbound}
    Let $\{Y_i\}_{i=1}^d$ be independent and identically distributed standard Gaussian random variables (i.e. $\mu = 0$ and $\sigma = 1$). Then letting $Z = \sum_{i=1}^d Y_i^2$ we have for any $\varepsilon > 0$
    \begin{gather*}
        \bbP{Z > \varepsilon} \leq 2^{d/2}e^{-\frac{\varepsilon}{4}}.
    \end{gather*}
\end{proposition}
\begin{proof}
    Since $Y_i \sim N(0,1)$, then $Z_i = Y_i^2$ is a chi-squared random variable with $1$ degree of freedom. By independence, then $Z = \sum_{i=1}^d Z_i$ is a chi-squared random variable with $d$ degrees of freedom. Then by Chernoff's inequality
    \begin{align*}
        \bbP{Z > \varepsilon} &= \bbP{e^{\lambda Z} > e^{\lambda \varepsilon}} \leq \frac{1}{e^{\lambda \varepsilon}}(1-2\lambda)^{-d/2}.
    \end{align*}
    We may choose any $\lambda < \frac12$ so we choose $\lambda = \frac14$ to obtain the bound
    \begin{gather*}
        \bbP{Z > \varepsilon} \leq 2^{d/2}e^{-\frac{\varepsilon}{4}}.
    \end{gather*}
    \end{proof}

    \begin{lemma}[Quadratic Newton tube]\label{lem:quadratic-newton-tube}
Suppose $\alpha > 1$. There exist constants $Q > \frac{L}{2m}$ and $p_{\mathrm{s}}>0$ with the following property: For every deterministic time $N$ and every state $X_N\in B_r(x^*)$,
\begin{gather}\label{eq:quadratic-success-probability}
    \bbP{\norm{X_{N+k}-x^*}\leq a_k\ \text{for all }k\geq 0\mid \mathcal F_N}
    \geq p_{\mathrm{s}}.
\end{gather}
where $a_0 = r$ and $a_{k+1} = Qa_k^2$ for $k \geq 1$. Consequently, with conditional probability at least $p_{\mathrm{s}}$, one has
\begin{gather*}
    X_{N+k}\to x^*
\end{gather*}
without leaving $B_r(x^*)$.
\end{lemma}
\begin{proof}
Fix $N \geq 1$. Let
\begin{gather*}
    C_0:=\frac{L}{2m},
    \qquad
    K_0:=\left(\frac{M_0}{2}\right)^\alpha.
\end{gather*}
Choose any $Q> C_0$ and set $a_0 = r$ and $a_{k+1} = Qa_k^2$. We shrink $r < r_0$ such that
\begin{gather*}
    Qr \leq 1\quad \text{and}\quad f(x) \leq \frac12,\quad \forall x \in B_r(x^*).
\end{gather*}
Consequently, $a_1 = Qa_0^2 \leq a_0 \leq r$. Applying induction yields that for all $k \geq 0$ we have $a_k \leq a_0^k$.

Thus, if
\begin{gather}\label{eq:quadratic-tube-noise-condition}
    K_0a_k^{2\alpha}\norm{Z_{N+k}}\leq a_{k+1}-C_0a_k^2,
\end{gather}
then on the event $\{\|X_{N+k} - x^*\| \leq a_k\}$ the inequalities \eqref{eq:quadratic-newton-estimate} and \eqref{eq:quadratic-tube-noise-condition} imply
\begin{align*}
    \norm{X_{N+k+1}-x^*}
    &\leq \norm{T(X_{N+k})-x^*}+f(X_{N+k})^\alpha\norm{Z_{N+k}}\\
    &\leq C_0a_k^2 + K_0a_k^{2\alpha}\|Z_{N+k}\| \\
    &\leq C_0a_k^2+a_{k+1}-C_0a_k^2
    =a_{k+1}.
\end{align*}
To guarantee \eqref{eq:quadratic-tube-noise-condition}, we require
\begin{gather*}
    \|Z_{N+k}\| \leq \frac{Q-C_0}{K_0}a_k^{2-2\alpha}.
\end{gather*}
Since $\alpha>1$, the exponent $2-2\alpha$ is negative, and hence the right-hand side tends to $\infty$ as $k\to\infty$. Set $R_k =: \frac{Q-C_0}{K_0}a_k^{2-2\alpha}$. By Proposition \ref{prop:tailbound} there exist constants $C_1,C_2 > 0$ such that
\begin{gather*}
    \bbP{\|Z_{N+k}\| > R_k} = \bbP{\|Z_{N+k}\|^2 > R_k^2} \leq C_1e^{-C_2a_k^{2-2\alpha}}.
\end{gather*}
It follows that
\begin{gather*}
    \sum_{k=0}^\infty \bbP{\norm{Z_{N+k}}>R_k}<\infty.
\end{gather*}
Consequently, there exists $I_0\in \mathbb{N}$ such that for all $i \geq I_0$
\begin{gather*}
    \bbP{\|Z_i\| \geq R_i} < \frac12,\quad \forall i \geq I_0.  
\end{gather*}
Using the inequality $\log(1-x) \geq -2x$ for all $0< x < \frac12$ gives
\begin{gather*}
    \sum_{k=I_0}^\infty \log \bbP{\|Z_k\| \leq R_k} \geq -2\sum_{k=I_0}^\infty \bbP{\|Z_k\| > R_k}  > -\infty.
\end{gather*}
Thus, adding back in the first $I_0-1$ terms and exponentiating the sum shows that
\begin{gather*}
    p_{\mathrm{s}}:=\prod_{k=0}^\infty \bbP{\norm{Z_k}\leq R_k}>0.
\end{gather*}
The independence of the $Z_n$ then implies that, conditionally on $\mathcal F_N$, the event
\begin{gather*}
    \mathcal S_N:=\bigcap_{k=0}^\infty \{\norm{Z_{N+k}}\leq R_k\}
\end{gather*}
has probability at least $p_{\mathrm{s}}$, i.e. on the event $\{X_N \in B_r(x^*)\}$ we have
\begin{gather*}
    \bbP{\mathcal{S}_N\ |\ \mathcal{F}_N} \geq p_s,\quad \forall N\geq 1. 
\end{gather*}
On $\mathcal S_N$, we apply induction to obtain
\begin{gather*}
    \norm{X_{N+k}-x^*}\leq a_k,
    \qquad k\geq 0.
\end{gather*}
Since $a_k\to 0$, this proves the lemma.
\end{proof}
\\
\begin{remark}\label{rem1}
    We note that $\alpha > 1$ appears quite naturally in the proof where we require the noise to generally be bounded, but at the cost of increasing the noise level, i.e. $\alpha > 1$. It is tempting to use a similar proof for $\alpha = 1$ as this implies the noise is on the same level as the Newton term. Unfortunately, the same proof strategy does not work for $\alpha = 1$ as this would require the $Z_k$ to stay within a bounded ball that is not shrinking.
\end{remark}\\

We shall also use the following simple observation. It handles the case where the process enters $B_r(x^*)$ and never leaves, even if it is not inside the particular shrinking tube constructed above.

\begin{lemma}[No exit implies local convergence]\label{lem:no-exit-local-convergence}
Let $\alpha > \frac12$ and $N$ be a deterministic or almost surely finite stopping time such that $X_N\in B_r(x^*)$. On the event
\begin{gather*}
    E_N:=\{X_{N+k}\in B_r(x^*)\ \text{for all }k\geq 0\},
\end{gather*}
one has
\begin{gather*}
    X_{N+k}\to x^*
    \qquad \text{almost surely on }E_N.
\end{gather*}
\end{lemma}
\begin{proof}
Fix $\varepsilon\in(0,r)$. We claim that the process cannot visit the annulus
\begin{gather*}
    A_\varepsilon:=\{x\in B_r(x^*):\norm{x-x^*}\geq \varepsilon\}
\end{gather*}
infinitely many times while staying forever inside $B_r(x^*)$.

Indeed, on $A_\varepsilon$ the function $f$ is bounded above and also bounded below by a positive constant. Thus the periodized Gaussian density of the next iterate has a strictly positive lower bound on $\T^d$, uniformly over all current states in $A_\varepsilon$. Since $\T^d\setminus B_r(x^*)$ has positive Lebesgue measure, there exists $\delta_\varepsilon>0$ such that
\begin{gather}\label{eq:annulus-exit-probability}
    \inf_{x\in A_\varepsilon}
    \bbP{X_{n+1}\notin B_r(x^*)\mid X_n=x}
    \geq \delta_\varepsilon.
\end{gather}
Therefore, each time the process lies in $A_\varepsilon$, it has conditional probability at least $\delta_\varepsilon$ of leaving $B_r(x^*)$ at the next step. 

Define the stopping times for $j > 1$:
\begin{align*}\label{eq:annulus-stopping-times}
    \sigma_1 &=: \inf\{k \geq 0: X_{N+k} \in A_\varepsilon\},\\
    \sigma_j &=: \inf\{k> \sigma_{j-1}: X_{N+k} \in A_\varepsilon\}.
\end{align*}
It follows from the strong Markov property and \eqref{eq:annulus-exit-probability}, we have for each $j$
\begin{gather*}
    \bbP{X_{N+\sigma_{j} + 1} \in B_r(x^*)\mid \mathcal{F}_{\tau_j}} \leq 1-\delta_\varepsilon.
\end{gather*}
Since $E_N \cap \{\sigma_1< \infty,\ldots, \sigma_J < \infty\} \subset \bigcap_{j=1}^J \{\sigma_j < \infty: X_{N + \sigma_j + 1} \in B_r(x^*)\}$ then
\begin{gather*}
    \bbP{E_N \cap \{\sigma_1< \infty,\ldots, \sigma_J < \infty\}} \leq (1-\delta_\varepsilon)^J.
\end{gather*}
Therefore, since $E_N \cap \{X_{N+k} \in B_r(x^*)\ i.o.\} = E_N \cap  \bigcap_{J\geq 1} \{\sigma_1< \infty,\ldots, \sigma_J < \infty\}$ then by continuity of probability we have
\begin{align*}
    \bbP{E_N \cap \{X_{N+k} \in B_r(x^*)\ i.o.\}} &= \lim_{J \to \infty} \bbP{E_N \cap \{\sigma_1< \infty,\ldots, \sigma_J < \infty\}} \\
    &\leq \lim_{J\to \infty}(1-\delta_\varepsilon)^J\\
    &= 0.
\end{align*}
Thus, on $E_N$, for every rational $\varepsilon\in(0,r)$, the iterates eventually lie outside $A_\varepsilon$. Equivalently,
\begin{gather*}
    \limsup_{k\to\infty}\norm{X_{N+k}-x^*}\leq \varepsilon
\end{gather*}
for every rational $\varepsilon>0$. Hence $\norm{X_{N+k}-x^*}\to 0$ almost surely on $E_N$.
\end{proof}

We can now combine the entrance mechanism with the local shrinking-tube mechanism.

\begin{theorem}[Global almost sure convergence on $\T^d$]\label{thm:main-periodic}
Assume \textbf{A1}--\textbf{A2}, and let $\{X_n\}$ be defined by \eqref{eq:torus-scheme} with $\alpha>1$. Then for every initial condition $X_0\in\T^d$,
\begin{gather*}
    X_n\to x^*
    \qquad \text{almost surely.}
\end{gather*}
Additionally, for almost every sample path $\omega$ there exists a random time $N$ such that the process $\{X_{N+k}(\omega)\}_{k=0}^\infty$ converges to $x^*$ quadratically.
\end{theorem}
\begin{proof}
Let $r$ be the radius chosen in \eqref{eq:choice-local-radius}. By Lemma \ref{lem:torus-entrance}, the process enters $B_r(x^*)$ almost surely. Define the successive entrance and exit times by
\begin{align*}
    \rho_1&:=\inf\{n\geq 0:X_n\in B_r(x^*)\},\\
    \eta_1&:=\inf\{n>\rho_1:X_n\notin B_r(x^*)\},
\end{align*}
and for $j\geq 1$,
\begin{align*}
    \rho_{j+1}&:=\inf\{n>\eta_j:X_n\in B_r(x^*)\},\\
    \eta_{j+1}&:=\inf\{n>\rho_{j+1}:X_n\notin B_r(x^*)\}.
\end{align*}
If $\eta_j<\infty$, then the strong Markov property and Lemma \ref{lem:torus-entrance} imply that $\rho_{j+1}<\infty$ almost surely. Thus every exit is followed almost surely by another entrance. On the event $\bigcup_{K \geq 1}\{\eta_{K} = \infty\}$, then Lemma \ref{lem:no-exit-local-convergence} yields the desired result.

If not, then at each entrance time $\rho_j$, Lemma \ref{lem:quadratic-newton-tube} gives a conditional probability at least $p_{\mathrm{s}}>0$ that the process remains in the shrinking tube
\begin{gather*}
    \norm{X_{\rho_j+k}-x^*}\leq a_k,
    \qquad k\geq 0.
\end{gather*}
where $a_k$ is as in the statement of the Lemma.

Define the event
\begin{gather*}
    G_j =: \{\|X_{\rho_j + k} - x^*\| > a_k,\ \forall k\}.
\end{gather*}
On $G_j^c$, by Lemma \ref{lem:quadratic-newton-tube} we have $X_n\to x^*$. On the other hand, the probability that the first $J$ entrances all fail to be successful is
\begin{gather*}
    \bbP{\bigcap_{j=1}^J G_j} \leq (1-p_s)^J.
\end{gather*}
Letting $J\to\infty$ we have
\begin{gather*}
    \bbP{\{X_n\ \text{never enters tube}\}} = \lim_{J \to \infty}\bbP{\bigcap_{j=1}^J G_j} \leq \lim_{J\to \infty}(1-p_s)^J = 0.
\end{gather*}
Hence, there exists a (random) index $K$ such that $\eta_K = \infty$ and $\bbP{G_K} = 1$. Thus convergence follows directly from Lemma \ref{lem:quadratic-newton-tube}.
\end{proof}
\\
\begin{remark}\label{rem2}
The shrinking-ball argument also gives a pathwise rate on the successful entrance event. Namely, on the event constructed in Lemma \ref{lem:quadratic-newton-tube} we obtain quadratic convergence
\begin{gather*}
    \norm{X_{N+k}-x^*}\leq a_k,
    \qquad k\geq 0,
\end{gather*}
where $a_k$ is as in the hypotheses of Lemma \ref{lem:quadratic-newton-tube}. Note that this is not the same as a deterministic rate for every sample path from time zero. The entrance time $N$ is random, and the proof only asserts that one of the random entrances is eventually successful almost surely.
\end{remark}\\

Observing the proof of Lemma \ref{lem:quadratic-newton-tube}, we can also generalize for $\alpha > \frac12$ at the cost of obtaining linear (or superlinear) convergence in the tube. This implies the result of Theorem \ref{thm:main-periodic} still holds even in this case.

\begin{lemma}[A shrinking Newton tube]\label{lem:shrinking-newton-tube}
There exist constants $\theta\in(0,1)$ and $p_{\mathrm{s}}>0$ with the following property. For every deterministic time $N$ and every state $X_N\in B_r(x^*)$,
\begin{gather}\label{eq:success-probability}
    \bbP{\norm{X_{N+k}-x^*}\leq r\theta^k\ \text{for all }k\geq 0\mid \mathcal F_N}
    \geq p_{\mathrm{s}}.
\end{gather}
Consequently, with conditional probability at least $p_{\mathrm{s}}$, one has
\begin{gather*}
    X_{N+k}\to x^*
\end{gather*}
without leaving $B_r(x^*)$.
\end{lemma}
\begin{proof}
Let
\begin{gather*}
    C_0:=\frac{L}{2m},
    \qquad
    K_0:=\left(\frac{M_0}{2}\right)^\alpha.
\end{gather*}
Choose $\theta\in(0,1)$ and, by decreasing $r$ if necessary, assume
\begin{gather}\label{eq:theta-radius-choice}
    C_0 r\leq \frac{\theta}{2}.
\end{gather}
Set
\begin{gather*}
    a_k:=r\theta^k,
    \qquad k\geq 0.
\end{gather*}
Suppose that $\norm{X_{N+k}-x^*}\leq a_k$. Since $a_k\leq r$, Proposition \ref{prop:torus-quadratic-bounds} gives
\begin{gather}\label{ineq:fbound}
    f(X_{N+k})^\alpha\leq \left(\frac{M_0}{2}\right)^\alpha\|X_{N+k}-x^*\|^{2\alpha} \leq  K_0a_k^{2\alpha}.
\end{gather}
Also, by Lemma \ref{lem:torus-newton-estimate},
\begin{gather}\label{ineq:Newtonbound}
    \norm{T(X_{N+k})-x^*}\leq \frac{L}{2m}\|X_{N+k} - x^*\|^2 \leq C_0 a_k^2.
\end{gather}
Thus, if
\begin{gather}\label{eq:tube-noise-condition}
    K_0a_k^{2\alpha}\norm{Z_{N+k}}\leq a_{k+1}-C_0a_k^2,
\end{gather}
then inequalities \eqref{ineq:fbound} and \eqref{ineq:Newtonbound} imply
\begin{align*}
    \norm{X_{N+k+1}-x^*}
    &\leq \norm{T(X_{N+k})-x^*}+f(X_{N+k})^\alpha\norm{Z_{N+k}}\\
    &\leq C_0a_k^2 + K_0a_k^{2\alpha}\|Z_{N+k}\| \\
    &\leq C_0a_k^2+a_{k+1}-C_0a_k^2
    =a_{k+1}.
\end{align*}
By \eqref{eq:theta-radius-choice},
\begin{gather*}
    a_{k+1}-C_0a_k^2
    =\theta a_k-C_0a_k^2 = (\theta - C_0 a_k)a_k \geq (\theta - C_0 r)a_k \geq  \frac{\theta}{2}a_k.
\end{gather*}
Therefore \eqref{eq:tube-noise-condition} is guaranteed if
\begin{gather}\label{eq:Z-threshold}
    \norm{Z_{N+k}}\leq R_k:=\frac{\theta}{2K_0}a_k^{1-2\alpha}.
\end{gather}
Since $\alpha>1/2$, the exponent $1-2\alpha$ is negative, and hence $R_k\to\infty$ as $k\to\infty$. The rest of the proof proceeds similarly to Lemma \ref{lem:quadratic-newton-tube}.
\end{proof}
\\
\begin{remark}\label{rem3}
We observe from the proof of Lemma \ref{lem:shrinking-newton-tube} that the threshold $\alpha>1/2$ is exactly the condition that the noise amplitude $f(x)^\alpha$ decays faster than the distance scale $\norm{x-x^*}$ near the minimizer, since locally $f(x)\sim \norm{x-x^*}^2$.
\end{remark}\\

\section{Conclusion and Future Directions}
In this manuscript, we analyze a randomized Newton algorithm \eqref{eq:torus-scheme} for optimization and prove convergence for multimodal functions defined on the torus with a unique global minimizer. Our algorithm exploits knowledge of the global minimum value to prove quadratic path wise almost sure convergence. The focus is on a setting that allows for a clean result.
 
Let us remark that our work can be extended in several directions. Instead of the classical Newton algorithm at the core, quasi-Newton methods like BFGS could be considered. This would be important when the Hessian is not available and in higher dimensions where evaluating the Hessian is computationally costly. Another potential generalization is to allow the minimum to be at a larger set and not just at a unique point. This is the generic case in overparametrized settings. The assumptions on the function near the minimum must then be adjusted and convergence measured in the distance to the minimizing set. Extension to problems where the minimum value is unknown requires major changes to the proofs and the convergence would probably be weaker than almost sure. While one can potentially replace $f_{min}$ in \eqref{eq:torus-scheme} with some estimator of the global minimum, the scheme loses the strong Markov property which is essential in the proofs of Lemmas \ref{lem:torus-entrance} and \ref{lem:no-exit-local-convergence}.

\section*{Funding Information}
This work was completed while the first author was funded under the National Defense Science and Engineering Graduate (NDSEG) Fellowship administered through the Department of the Air Force (AFRL/Space Force).

\vskip2mm

%\par{\bf References.}%\, Always use $\backslash$cite$\{biblabelname\}$ (eg. \cite{taubes1}) to cite
          %references which have been named in the bibliography via
          %$\backslash$bibitem$\{biblabelname\}$.

          % Non-BibTeX users please use

          \end{document}